\documentclass[a4paper,12pt,reqno]{amsart}
\usepackage[utf8]{inputenc}
\usepackage{amssymb,amsthm}
\usepackage{enumerate}
\usepackage[margin=3cm]{geometry}
\usepackage{graphicx}
\usepackage{color}

\renewcommand{\epsilon}{\varepsilon}

\newtheorem{theorem}{Theorem}[section]
\newtheorem{prop}[theorem]{Proposition}
\newtheorem{cor}[theorem]{Corollary}
\newtheorem{lemma}[theorem]{Lemma}
\newtheorem{remark}[theorem]{Remark}

\newcommand{\R}{\mathbb{R}}

\newcommand{\N}{\mathbb{N}}
\newcommand{\Z}{\mathbb{Z}}

\newcommand{\mc}{\mathcal}
\newcommand{\ep}{\varepsilon}

\def\Res{\operatorname{Res}}

\newenvironment{sistema}
{\left\lbrace\,\begin{array}{@{}l@{}}} {\end{array}\right.}

\title[A new approach for the study of limit cycles]{A new approach for the study of limit cycles}

\author{J.D. Garc\'{\i}a-Salda\~{n}a}
\address{Departamento de Matem\'{a}tica y F\'isica Aplicadas,
Universidad Cat\'olica de la Sant\'isima Concepci\'on, Alonso de Ribera 2850, Concepci\'on, Chile.}
\email{jgarcias@ucsc.cl}

\author{A. Gasull}
\address{Departament de Matem\`{a}tiques,
Universitat Aut\`{o}noma de Barcelona, Edifici C 08193 Bellaterra, Barcelona, Spain.} \email{gasull@mat.uab.cat}

\author{H. Giacomini}
\address{ Institut Denis Poisson. Universit\'{e} de Tours,  C.N.R.S. UMR 7013. 37200 Tours, France.} \email{Hector.Giacomini@lmpt.univ-tours.fr}

\date{}

\begin{document}

\begin{abstract}
We prove that  star-like limit cycles of any planar polynomial system can also be seen either as solutions
defined on a given interval of a new associated planar non-autonomous polynomial
system or as heteroclinic solutions of a 3-dimensional polynomial system. We illustrate these points of view with several examples.  One of the key ideas in our approach is to decompose the periodic solutions as the sum of two suitable functions. As a first application we use these new approaches to prove that all star-like reversible limit cycles are algebraic. As a second application   we introduce 
 a function whose zeroes control the periodic orbits that persist as limit cycles
when we perturb a star-like reversible center. As far as we know  this is the first time that this question is solved in full generality. Somehow, this function plays a similar role that an
Abelian integral for studying perturbations of Hamiltonian systems.

\end{abstract}

\maketitle

\section{Introduction and  results}

Consider a planar differential system
\begin{align}\label{sis:pla}
\dot{x}=X(x,y),\qquad
\dot{y}=Y(x,y),
\end{align}
where $X$ and $Y$ are polynomial functions vanishing at the origin.
The most elusive problem about this system is to know its number of
limit cycles. Many efforts have been dedicated to this objective and
the reader can consult several books, and their references, where
this question is addressed, see \cite{Llibre,Yakovenko,Ye,Zhang}. Of
course, it is also very related with the celebrated XVIth Hilbert's
problem, see \cite{Ilyashenko,JibinLi}.

The goal of this work is to present a new approach to study the so
called star-like limit cycles. To give a first description of
our approach we start introducing some notation.

By applying the polar coordinates transformation $x=r\cos\theta$,
$y=r\sin\theta$ to system~\eqref{sis:pla} and by eliminating the
variable $t$, we obtain the associated non-autonomous differential
equation
\begin{equation}\label{ecpol}
\frac{{\rm d}r}{{\rm d}\theta}=\frac{R\big(r\cos\theta,r\sin\theta\big)}{T\big(r\cos\theta,r\sin\theta\big)},
\end{equation}
where
\begin{align*}
R(x,y)=\big(x X(x,y)+ y Y(x,y)\big )/r \quad \mbox{and}\quad T(x,y)=\big(x Y(x,y)- y
X(x,y)\big)/r^2.
\end{align*}
Notice that \eqref{ecpol}  is well defined in the region
$\mathbb{R}^2\setminus \mathcal{T},$ where $\mathcal{T}=
\{(x,y)\,:\,T(x,y) =0\}.$ Moreover, the periodic solutions of
\eqref{sis:pla} that do not intersect $\mathcal T$ correspond to
smooth $2\pi$-periodic solutions of \eqref{ecpol}. These periodic
orbits are usually called {\it star-like} periodic orbits. In
particular when these $2\pi$-periodic solutions are isolated in the
set of periodic solutions they will correspond to star-like
limit cycles.

As we will prove in Theorem~\ref{pr:fg} and Corollary~\ref{co:fg} each smooth star-like periodic orbit
$r=F(\theta)$ can be written
uniquely as
\begin{equation}\label{r:cs}
r=F(\theta)= f(\sin \theta)+g(\sin\theta)\cos \theta,
\end{equation}
for some $f,g\in \mathcal{C}^\infty([-1,1]),$ that is, both functions are $\mathcal{C}^\infty$  in an open  neighborhood of $[-1,1].$   A similar decomposition, but involving the functions $\sin(2\pi t/T)$ and $\cos(2\pi t/T),$ can also be applied to each of the components of any $T$-periodic solution of a $k$-dimensional differential system, parameterized by the time $t$. We believe that this point of view will help to a better understanding of this type of solutions. In this paper we will concentrate on the planar case.

The starting point of this work is to use equation \eqref{r:cs}
and an associated non-autonomous polynomial differential system
involving $u,f(u)$ and $g(u)$ to characterize the  star-like limit cycles of
system \eqref{sis:pla} as the solutions with interval of definition including  all $u$ in $[-1,1].$ These solutions can also be seen as heteroclinic solutions of a $3d$-autonomous system joining singularities contained in the planes $u=-1$ and  $u=1,$ see Theorems~\ref{teo:sisf'g'} and~\ref{teo} in next section.    As applications of this approach:

\begin{itemize}

\item We give a systematic procedure to find reversible
star-like limit cycles with respect to a straight line and prove that all of  them are algebraic, see Theorem~\ref{te:stl}.

\item We introduce a function, similar to an Abelian integral, that
controls which of the periodic orbits of a reversible star-like
center can persist after a perturbation of the system, see Theorem
\ref{te:N(a)}.
\end{itemize}

Let us describe with more detail our main results given in the above two items. 

A planar system is called {\it reversible (with respect to a straight line)} if 
after a rotation it is  invariant by the change of
variables and time $(x,y,t)\rightarrow (-x,y,-t).$ These vector fields can be written as
\[
\dot{x}= A(x^2,y),\qquad\qquad \dot{y}=x B(x^2,y),
\]
for some smooth functions $A$ and $B.$ Similarly, a limit cycle is called {\it reversible} if after a translation and a rotation writes as  $(x(t),y(t))$ and it satisfies that  $(x(t),y(t))=(-x(-t),y(-t))$ for all $t\in\R.$   Recall also that a limit cycle is called algebraic if it is included in the zero set of a real polynomial in two variables, that is in $G(x,y)=0,$ for some polynomial $G.$ In Section~\ref{se:rlc}  we prove:
 
 \begin{theorem}\label{te:stl}{Star-like reversible limit cycles of polynomial systems are algebraic.}\end{theorem}

An auxiliary result that  we use to prove the above theorem is given in the  self-contained Section~\ref{se:com}, where we develop a method for detecting whether two different polynomial systems share a  trajectory.

About our second application, consider any perturbed reversible system:
\begin{align}\label{sRev_per}
\dot{x}= A(x^2,y)+\ep C(x,y),\qquad\qquad \dot{y}=x B(x^2,y)+\ep
D(x,y),
\end{align}
where we  assume that in the punctured neighborhood  $\mathcal{U}$ of the origin, the function  $x^2B(x^2,y)-y
A(x^2,y),$ that gives the sign of $\dot\theta$, does not vanish and $C(0,0)=D(0,0)=0.$ This set $\mathcal{U}$ is non empty for instance when
 $A(0,y)=ay+O_2(y),$ where $O_2(y)$ denotes
terms of order at least 2 in $y,$ and $B(0,0)=b$, $ab<0,$ because
then the origin is a
 monodromic non-degenerate critical point. Notice that under these
 hypotheses and in this region, when $\ep=0$ the origin is a
 star-like reversible center.

To state the results of our second application we also need the notation  given in the following lemma. 

\begin{lemma}\label{le:des} Introduce the  variables $f,g,u,v,$ with $v^2=1-u^2,$ and consider
	\[
	x=(f+vg)v,\quad y=(f+vg)u.
	\]
	Then, any polynomial $H(x,y)$
	writes, using that $v^{2k+1}=v(1-u^2)^k$, as
	\begin{equation*}
	H(x,y)=H\big((f+vg)v,(f+vg)u)\big)=H_0(f,g,u)+v H_1(f,g,u).
	\end{equation*}
	Moreover, when $H(x,y)=K(x^2,y)$ for some polynomial $K,$
\[H_0(f,g,u) =K_{0,0}(f,u)+ (1-u^2) g^2 
K_{0,1}(f,g,u)\mbox{ and } H_1(f,g,u) =g K_{1,0}(f,g,u),\] for some
polynomials $K_{0,0}, K_{0,1}$ and $K_{1,0},$ where
$K_{0,0} (f,u)=K\big((1-u^2)f^2,u f\big).$	
\end{lemma}

As we will see in Lemma~\ref{le_f'rev} the solutions of the unperturbed system~\eqref{sRev_per}$_{\ep=0}$   with initial condition $r(0)=\rho$ are  $r=f_\rho(\sin
\theta)$ where $f_\rho(u)$ is the solution of the first order
Cauchy initial value problem
\begin{equation}\label{f'_rev}
f^{\prime}=\frac{\left(A_{0,0}+u  B_{0,0} f \right)
	f}{(1-u^2)B_{0,0} f -uA_{0,0}}, \quad f(0)=\rho,
\end{equation}
where  $A_{0,0}=A\big((1-u^2)f^2,uf\big)$ and
$B_{0,0} =B\big((1-u^2)f^2,u f\big)$.

Our main result is:
\begin{theorem}\label{te:N(a)}
For each $\rho>0,$ such that $r=f_{\rho}(\sin\theta)$ is periodic orbit  of the unperturbed system~\eqref{sRev_per}$_{\ep=0},$ we define
\begin{equation}\label{N(a)}
N(\rho)=\int_{-1}^1M_{\rho}(u){\rm e}^{-\int_0^u L_{\rho}(s){\rm d}s}{\rm d}u,
\end{equation}
where
\begin{equation*}
L_\rho(u)=\frac{S_\rho(u)}{(1-u^2)(uA_{0,0}-(1-u^2)B_{0,0}f_{\rho}(u))^2},
\end{equation*}
 with
\begin{multline*}
S_{\rho}(u)=(1-u^2)\big(A_{1,0}B_{0,0}-A_{0,0}B_{1,0}+2u(1-u^2)B_{0,0}^2\big)f_{\rho}^2(u)
\\-4u^2(1-u^2) A_{0,0}B_{0,0}f_{\rho}(u)-u(1-2u^2)A_{0,0}^2
\end{multline*} and
\begin{equation*}
M_{\rho}(u)=\frac{-f_{\rho}(u)\big(A_{0,0}{D}_0-(1-u^2)B_{0,0}{C}_1
	f_{\rho}(u)\big)}{(1-u^2)\big(uA_{0,0}-(1-u^2)B_{0,0}f_{\rho}(u)\big)^2},
\end{equation*}
 where all the functions $D_0, C_1,A_{0,0}, A_{1,0},B_{0,0}$ and $B_{1,0},$ introduced in Lemma~\ref{le:des}, are evaluated  either at  $(f,g,u)=(f_\rho(u),0,u)$ or at  $(f,u)=(f_\rho(u),u)$.
If  for $\ep>0$ small enough, there is a continuous family $r=r(\theta,\ep)$ of limit cycles of~\eqref{sRev_per} such that $r(0,\ep)\to \rho_0$ as $\ep$ goes to zero, then $N(\rho_0)=0$.
\end{theorem}

Recall that Abelian integrals are used to study the same problem but
for perturbations of Hamiltonian systems and that this approach can be
easily extended to study the perturbations of systems with an
integrating factor, see for instance~\cite{CL}. Theorem~\ref{te:N(a)} is similar but for perturbations of reversible centers.
As far as we know, this is the first time that this
question is solved in full generality for systems having a star-like
reversible center.  We are studying if simple zeroes of $N$ always give rise to limit cycles of the perturbed system.

In Section~\ref{se:abe} we apply Theorem~\ref{te:N(a)} to several examples.

\section{General results}

We start proving the  general decomposition result stated in~\eqref{r:cs} for analytic $2\pi$-periodic solutions that, as we have already said, is one of the key points of our approach.  Next lemma shows that the functions $f$ and $g$ are analytic in $(-1,1).$

\begin{lemma}\label{lem:ana}
	Let $F(\theta)$ be a real analytic $2\pi$-periodic function.  Then: 
	\begin{enumerate}[(i)] 
		\item There exist two unique analytic functions  $f$ and $g$ defined in $(-1,1)$   such that  
		\begin{equation*}
			F(\theta)= f(\sin \theta)+g(\sin\theta)\cos \theta.
		\end{equation*}
		
		\item There exist two unique analytic functions  $\bar f$ and $\bar g$ defined in $(-1,1)$   such that  
		\begin{equation}\label{r:cs22}
		F(\theta)= \bar f(\cos \theta)+\bar g(\cos\theta)\sin \theta.
		\end{equation}
	\end{enumerate}
\end{lemma}

\begin{proof} 	We will prove item $(ii)$. It is easy to see that both are equivalent, simply by changing $\theta$ by $\theta+\pi/2.$
	
Notice that evaluating \eqref{r:cs22} at $-\theta$ we get that
\[
F(-\theta)=\bar f\big(\cos (-\theta)\big)+\bar g\big(\cos (-\theta)\big)\sin (-\theta)=\bar f(\cos \theta)-\bar g(\cos\theta)\sin \theta.
\]	
Joining this equation and \eqref{r:cs22}  we obtain that
\[
\bar f(\cos \theta)=\frac{F(\theta)+F(-\theta)}2\quad\mbox{and}\quad \bar g(\cos \theta)=\frac{F(\theta)-F(-\theta)}{2\sin \theta}.
\] 	
For $\theta\in(0,\pi),$ $\cos \theta$ takes all values in $(-1,1)$ and doing the change of variables $u=\cos\theta$ we get that
\[
\bar f(u)=\frac{F(\arccos u)+F(-\arccos u)}2,\quad \bar g(u)=\frac{F(\arccos u)-F(-\arccos u)}{2\sqrt{1-u^2}}.
\] 
Clearly, both functions are analytic in $(-1,1).$	
\end{proof}

Next, we introduce some preliminary results that will be used to show that the functions $f,g,\bar f$ and $\bar g$ in the above decompositions are $\mathcal{C}^\infty([-1,1])$.

Recall that  the  Chebyshev polynomials of first and second kind, are respectively, $T_n$ and $U_n$ where
\begin{equation}\label{eq:che}
\cos (n\theta)=T_n(\cos\theta),\quad\quad \sin (n\theta)= U_{n-1} (\cos \theta) \sin \theta,
\end{equation}
and the polynomials $T_n$ and $U_n$ both satisfy the recurrence
\begin{equation}\label{eq:eed}
P_{n+1}(x)= 2 x P_n(x) - P_{n-1}(x),
\end{equation}
with initial conditions  $P_0=1$ and $P_1=x,$  (first kind and $T_n=P_n$) or
$P_0=1$ and $P_1=2x,$ (second kind and $U_n=P_n$).  Given any smooth function $F$, $F^{(k)}$ denotes its $k$-th derivative.

Given any polynomial $P,$  define $\|P\|:=\max_{x\in[-1,1]} |P(x)|.$  
Let $\mathcal{P}_n$ denote the set of polynomials of degree $\le n.$ Markov inequality for  polynomials says that for any $0<k\in\N,$
\begin{equation}\label{eq:mar}
	\|P^{(k)}\|\le \|T^{(k)}_n\|\,\|P\|,\quad P\in\mathcal{P}_n
	\end{equation}
and 
\[
	\|T^{(k)}_n\|=T^{(k)}_n(1)=\frac{n^2(n^2-1^2)(n^2-2^2)\cdots(n^2-(k-1)^2)}{1\cdot 3 \cdots (2k-1)}.
	\]
It was proved by Andrei Markov in 1889 for $k=1$ ($\|P^{\prime}\|\le n^2\,\|P\|$) and extended to any $k\ge1$ by his kid brother Vladimir Markov in 1892, see \cite{Sha}.

We will use that
\begin{equation}\label{eq:in1}	
\|T^{(k)}_n\|\le \frac{n^{2k}}{(2k-1)!!}
	\end{equation}
and, moroever, since $U_{n-1}= T_n^\prime /n$
\begin{equation*}
	\|U^{(k)}_{n-1}\|=\frac1n\|T^{(k+1)}_n\|\le \frac{n^{2k+1}}{(2k+1)!!}.
	\end{equation*}
	
We will also need the following version of Bernstein theorem for analytic $2\pi$-periodic functions, see Exercise 4 in \cite[Sec. 4]{Ka}.

\begin{theorem}\label{te:ana} Let $F$ be a $2\pi$-periodic analytic function. Then there exist two positive constants $K$ and $a$ such that its $n$-th Fourier coefficients $F_n, n\in\Z,$ satisfies
	\[
	|F_n|\le K {\rm e}^{-a |n|}.
	\]	  
\end{theorem}

Next proposition is our  last preliminary result.  Its proof is a direct consequence of the dominated convergence theorem.

\begin{prop}\label{pr:dct}
Fix $k\in\N$ and let $\mathcal{U}\subset\R$ be an open interval. Assume that for all $n\in\N\cup\{0\}$ the functions $G_n:\mathcal{U}\to\R$ are $\mathcal{C}^k(\mathcal{U}),$ and for each $x_0\in\mathcal{U}$ there is an open subset of $\mathcal U,$  $\mathcal{V}_{x_0,k}\ni x_0,$ such that
\[
 \big|G^{(k)}_n(x)\big|_{\{x\in \mathcal{V}_{x_0,k}\}}\le M_{n,k}\quad \mbox{and}\quad \sum_{n=0}^\infty M_{n,k}<\infty. 
\]	
Then, if  $G(x)=\sum_{n=0}^\infty G_n(x),$ it holds that  $G\in\mathcal{C}^k(\mathcal{U})$ and
$G^{(k)}(x)=\sum_{n=0}^\infty G_n^{(k)}(x).$	
\end{prop}

\begin{proof}
	We will prove the result for $k=1.$ The result for $k>1$ simply follows by applying the same proof to $G_n^{(k-1)}.$  
	
	Consider $k=1$. Let us prove first that $G$ is derivable. By the mean value theorem, we have that for $|h|$ small enough
	\[
	\frac{G(x_0+h)-G(x_0)}{h}=\sum_{n=0}^\infty \frac{G_n(x_0+h)-G_n(x_0)}{h}=\sum_{n=0}^\infty G_n^{\prime}(s_{n}),
	\]
for some  $s_n\in(x-|h|,x+|h|)	\subset \mathcal{V}_{x_0,1}.$ Hence 
\[
\left|\frac{G(x_0+h)-G(x_0)}{h}\right|= \left|\sum_{n=0}^\infty G_n^{\prime}(s_{n})\right|\le\sum_{n=0}^\infty M_{n,1}
<\infty
\]
and  by the dominated convergence theorem,
\begin{align*}
G^\prime(x_0)&=\lim_{h\to 0} \frac{G(x_0+h)-G(x_0)}{h} =\lim_{h\to 0}\sum_{n=0}^\infty \frac{G_n(x_0+h)-G_n(x_0)}{h}\\&=\sum_{n=0}^\infty \lim_{h\to 0}\frac{G_n(x_0+h)-G_n(x_0)}{h}= \sum_{n=0}^\infty G_n^\prime(x_0).
\end{align*} 
That $G'$ is continuous is again a consequence of the dominated convergence theorem, because the fact that $|G_n^\prime(x)|\le M_{n,1}$ for all $x\in \mathcal{V}_{x_0},$  and $\sum_{n=0}^\infty M_{n,1}<\infty$ implies that
\[
\lim_{h\to 0}G^\prime(x_0+h)=\lim_{h\to 0} \sum_{n=0}^\infty G_n^\prime(x_0+h)= \sum_{n=0}^\infty \lim_{h\to 0} G_n^\prime(x_0+h)=\sum_{n=0}^\infty G_n^\prime(x_0)=G'(x_0).
\]
\end{proof}	
	
Next result gives the desired decomposition. 

\begin{theorem}\label{pr:fg}
	Let $F(\theta)$ be a real analytic $2\pi$-periodic function.  Then: 
	\begin{enumerate}[(i)] 
		\item There exist two unique   $\mathcal{C}^{\infty}([-1,1])$ functions $f$ and $g$  such that  
		\begin{equation*}
		F(\theta)= f(\sin \theta)+g(\sin\theta)\cos \theta.
		\end{equation*}
		
		\item There exist two unique   $\mathcal{C}^{\infty}([-1,1])$ functions $\bar f$ and $\bar g$  such that  
		\begin{equation}\label{r:cs2}
		F(\theta)= \bar f(\cos \theta)+\bar g(\cos\theta)\sin \theta.
		\end{equation}
		
	\end{enumerate}
\end{theorem}

\begin{proof} 	As in the proof of Lemma~\ref{lem:ana} it suffices to prove item $(ii)$. The uniqueness is a consequence of this lemma.

Let us prove~\eqref{r:cs2}. We start decomposing the Fourier series of $F$ as follows 
\begin{align*}
F(\theta)&= \sum_{n=0}^\infty a_n \cos( n\theta)+ \sum_{n=0}^\infty b_n \sin( n\theta)\\&= 
\sum_{n=0}^\infty a_n T_n(\cos \theta)+ \sum_{n=1}^\infty b_n U_{n-1}(\cos \theta)\sin \theta\\&=
\sum_{n=0}^\infty a_n T_n(\cos \theta)+ \sin \theta \sum_{n=1}^\infty b_n  U_{n-1}(\cos \theta),
\end{align*}
where, we have used~\eqref{eq:che}, and since $F$ is analytic, by Theorem \ref{te:ana},
\begin{equation}\label{eq:in3}
|a_n|\le M {\rm e}^{-a n},\quad |b_n|\le M {\rm e}^{-a n} ,
\end{equation}
for some $M>0$ and $a>0.$	
Hence, we introduce  the functions
\[
\bar f(u):= \sum_{n=0}^\infty a_n T_n(u)\quad\mbox{and}\quad \bar g(u):=\sum_{n=1}^\infty b_n U_{n-1}(u),
\]
and we prove first that they are of class $\mathcal{C}^{\infty}(-1,1).$ 	Consider for instance $\bar f.$   Although we already know  that $\bar f$ is analytic in $(-1,1),$ see Lemma~\ref{lem:ana}, we include a different proof that it is $\mathcal{C}^{\infty}(-1,1)$ because it gives the idea of how to proceed at $-1$ and $1.$ By joining inequalities ~\eqref{eq:in1} and~\eqref{eq:in3} it holds that for $k\in\N$ 
\begin{align*}
\Big\| \sum_{n=0}^\infty a_n T_n^{(k)}(u) \Big\|&\le \sum_{n=0}^\infty \big\|a_n T_n^{(k)}(u) \big\|\le  \sum_{n=0}^\infty M {\rm e}^{-a n} \frac{n^{2k}}{(2k-1)!!} \\&\le \frac {M}{(2k-1)!!}\sum_{n=0}^\infty \frac  {n^{2k}}{{\rm e}^{a n}}<\infty.
\end{align*}	
By applying Proposition~\ref{pr:dct} to $G_n=a_nT_n$ for all $k\in\N,$ we get that  $\bar f\in\mathcal{C}^{\infty}(-1,1).$ 

In order to apply Proposition~\ref{pr:dct} to prove that $\bar f$ is  $\mathcal{C}^{\infty},$ at the points $\pm 1$ we need to bound all the derivatives of $T_n$ on a fixed neighborhood of each of these points. We consider, for instance $u=1$ and define $J_\ep:=[1-\ep,1+\ep] ,$ for $\ep>0$ to be fixed. We will bound $\max_{y\in J_\ep} | T_n^{(k)}(y)|.$

With this goal, we introduce $Q_{n,\ep,k}(u):= {\ep^{-k}} T_n(\ep u+1).$ It is clear that
\[
 Q_{n,\ep,k}^{(k)}(u)= T_n^{(k)}(\ep u+1) 
\]
and
\[
\max_{y\in J_\ep} | T_n^{(k)}(y)|=\max_{u\in[-1,1]} | T_n^{(k)}(\ep u +1)|= \| T_n^{(k)}(\ep u +1)\|= \|  Q_{n,\ep,k}^{(k)}\|.
\]
By applying Markov inequality \eqref{eq:mar} to $Q_{n,\ep,k}$ we get that for any $k\in\N,$
\[
	\|Q_{n,\ep,k}^{(k)}\|\le \|T^{(k)}_n\|\,\|Q_{n,\ep,k}\|,
\]
or equivalently,
\begin{equation}\label{eq:max0}
\max_{y\in J_\ep} | T_n^{(k)}(y)|\le \frac{\|T^{(k)}_n\|}{\ep^k}\,\|T_{n}(\ep u +1)\|.
\end{equation}
Since $\|T_n\|=1,$  $T_n'(1)=n^2$ and all roots of $T_n$ are real and contained in $(-1,1)$ it holds that $|T_n^\prime(u)|_{u\ge1}>0$ and as a consequence
\begin{equation}\label{eq:max}
\|T_{n}(\ep u +1)\|=\max_{y\in J_\ep} |T_n(y)|=T_n(1+\ep).
\end{equation}
Let us compute this last expression. If $t_n:=t_n(\ep)=T_n(1+\ep),$ by \eqref{eq:eed}  it holds that
\[
t_{n+1}= 2 (1+\ep) t_n - t_{n-1},\quad t_0=1\quad\mbox{and}\quad t_1=1+\ep.
\]
 Either by solving the above second order linear difference equation with constant coefficients, or by using the expression of the Chebyshev polynomials in terms of square roots,  we get that
\[
t_n=T_n(1+\ep)=\frac12\left(1+\ep +\sqrt{\ep(2+\ep)}\right)^n+\frac12\left(1+\ep-\sqrt{\ep(2+\ep)}\right)^n,
\]
and as a consequence,
\[
|T_n(1+\ep)|=|t_n|\le \big|1+\ep +\sqrt{\ep(2+\ep)}\big|^n.
\]
By using the above inequality,  \eqref{eq:in3}, \eqref{eq:max0} and \eqref{eq:max} we obtain that 
\begin{equation*}
\max_{y\in J_\ep} | T_n^{(k)}(y)|\le \frac{n^{2k}}{\ep^k(2k-1)!!} \big|1+\ep +\sqrt{\ep(2+\ep)}\big|^n.
\end{equation*}
Taking $y\in J_\ep,$ by using the above inequality and \eqref{eq:in3} we get that
\begin{align*}
	\max_{y\in J_\ep}\Big| \sum_{n=0}^\infty a_n T_n^{(k)}(y) \Big|&\le \sum_{n=0}^\infty 	\max_{y\in J_\ep} \big|a_n T_n^{(k)}(y) \big|\\&\le  \sum_{n=0}^\infty M {\rm e}^{-a n} \frac{n^{2k}}{\ep^k(2k-1)!!} \big|1+\ep +\sqrt{\ep(2+\ep)}\big|^n\\&\le \frac {M}{\ep^k(2k-1)!!}\sum_{n=0}^\infty n^{2k} \big|{\rm e}^{-a}(1+\ep +\sqrt{\ep(2+\ep)})\big|^n  , 
\end{align*}
which is a convergent series if $\ep>0$ is so small that $ {\rm e}^{-a}(1+\ep +\sqrt{\ep(2+\ep)})<1.$	
Again, as a consequence of Proposition~\ref{pr:dct} applied to $G_n=a_nT_n$ and all $k\in\N,$ we get that  $\bar f$ is $\mathcal{C}^{\infty}$ at $u=1.$  The proof at $u=-1$ is analogous.  The proof for $\bar g$ also follows the same steps.
\end{proof}

A consequence of Theorem~\ref{pr:fg} is:

\begin{cor}\label{co:fg} Any star-like limit cycle $r=F(\theta)$ of \eqref{sis:pla} can be uniquely written as
	\[
	r=F(\theta)=f(\sin \theta)+g(\sin\theta)\cos \theta,
	\]	
	where $f$ and $g$ are $\mathcal{C}^\infty ([-1,1]).$
\end{cor}

\begin{proof}
Since star-like limit cycles are $r=F(\theta)$ where this function is a solution of~\eqref{ecpol} and in a neighborhood of it the denominator of this equation does not vanish, $F$ is a $2\pi$-periodic  analytic function. Then the result follows by using  Theorem~\ref{pr:fg}. 
\end{proof}

Our first characterization of the periodic orbits of~\eqref{sis:pla} as  solutions of a planar non-autonomous system is given in next theorem.

\begin{theorem}\label{teo:sisf'g'}
	Let $r=f(\sin \theta)+g(\sin \theta)\cos \theta$ be a star-like periodic orbit of \eqref{ecpol}
	where
	$f,\,g\in\mathcal{C}^\infty([-1,1])$. Then $f(u)$ and $g(u)$ are solutions of 
	\begin{equation}\label{sis:f'g'}
	\begin{cases}
	(1-u^2)T_1f^\prime(u)+(1-u^2)T_0g^\prime(u)-uT_0g(u)-R_0=0,\\
	T_0f^\prime(u)+(1-u^2)T_1g^\prime(u)-uT_1g(u)-R_1=0,
	\end{cases}
	\end{equation}
	where  $T$ and $R$ are given in the polar expression~\eqref{ecpol} of~\eqref{sis:pla} and  $T_i=T_i(f(u),g(u),u)$ and $R_i=R_i(f(u),g(u),u)$, for
	$i=0,1,$ as in Lemma~\ref{le:des}.
\end{theorem}

Notice that here, and from now on, $T_0$ and $T_1$ do not denote the first two  Chebyshev polynomials, but the polynomials associated to the decomposition of~$T.$

\begin{proof}[Proof of Theorem~\ref{teo:sisf'g'}] The right-hand side of the polar expression ~\eqref{ecpol} of~\eqref{sis:pla}, replacing
\begin{align*}
x&=r\cos\theta= \big(f(\sin \theta)+g(\sin\theta )\cos \theta\big)\cos \theta=(f(u)+g(u)v)v,\\
y&=r\sin\theta= \big(f(\sin \theta)+g(\sin\theta )\cos \theta\big)\sin \theta=(f(u)+g(u)v)u,
\end{align*}	
	 can be written as   
	$$
 \frac{R(r\cos\theta,r\sin\theta)}{T(r\cos\theta,r\sin\theta)}=\frac{R((f(u)+g(u)v)v,(f(u)+g(u)v)u)}{T((f(u)+g(u)v)v,(f(u)+g(u)v)u)}.
$$
Any star-like periodic  orbit of~\eqref{sis:pla} $r=f(\sin \theta)+g(\sin \theta)\cos \theta$ must satisfy
\begin{align*}
\frac{{\rm d}r}{{\rm d}\theta}&=f^{\prime}(\sin \theta)\cos\theta+g^{\prime}(\sin\theta)\cos^2\theta-g(\sin\theta)\sin\theta
\\&=f^\prime(u)v+g^\prime(u)v^2-g(u)u\\&=f^\prime(u)v+g^\prime(u)(1-u^2)-g(u)u.
\end{align*}
Hence
\[
f^\prime(u)v+g^\prime(u)(1-u^2)-g(u)u=\frac{R_0(f(u),g(u),u)+R_1(f(u),g(u),u)v}{T_0(f(u),g(u),u)+T_1(f(u),g(u),u)v}.
\]
Eliminating the denominator and  using once more that $v^2=1-u^2,$ we arrive to
 \begin{align*}
	(1-u^2)\big(f^\prime(u)T_1&+g^\prime(u)T_0\big)-ug(u)T_0
	-R_0+\\&\Big(T_0f^\prime(u)+\big((1-u^2)g^\prime(u)-ug(u)\big)T_1-R_1\Big)v=0,
	\end{align*}
where $T_i=T_i(f(u),g(u),u)$, $R_i=R_i(f(u),g(u),u)$, $i=0,1$.
From the above equality and by using the uniqueness of decomposition of \eqref{r:cs} proved in Theorem~\ref{pr:fg} both terms must be zero, giving rise to the system of two differential equations~\eqref{sis:f'g'}.	
\end{proof}

The following theorem gives a new and  geometric interpretation of the star-like periodic orbits of system~\eqref{sis:pla}.

\begin{theorem}\label{teo}
	Let $r=f(\sin \theta)+g(\sin \theta)\cos \theta$ be a star-like periodic orbit of \eqref{ecpol}
	where
	$f,\,g\in\mathcal{C}^\infty([-1,1])$.
	Then    $\gamma(s)=\big(f(u(s)),g(u(s)),u(s)\big)$ is a heteroclinic solution
 of the following $3d$-polynomial differential system 
	\begin{equation}\label{sis:3d}
	\begin{cases}
	\dfrac {{\rm d}f}{{\rm d}s}=(1-u^2)(R_1T_0-R_0T_1),\\[0.2cm]
	\dfrac {{\rm d}g}{{\rm d}s}=-(1-u^2)T_1(R_1+ugT_1)+T_0(R_0+ugT_0),\\[0.2cm]
	\dfrac {{\rm d}u}{{\rm d}s}=(1-u^2)(T_0^2-(1-u^2)T_1^2),
	\end{cases}
	\end{equation}
	 joining two critical points, each one of them in one of the two invariant planes $u=1$ and $u=-1,$ where
	 $T_i=T_i(f,g,u)$ and $R_i=R_i(f,g,u)$ are as in Theorem~\ref{teo:sisf'g'}.
	\end{theorem}

\begin{proof} By isolating $f'$ and $g'$ from system~\eqref{sis:f'g'} of Theorem~\ref{teo:sisf'g'}, we get 
\begin{align*}
f^\prime(u)&=\frac{R_1T_0-R_0T_1}{T_0^2-(1-u^2)T_1^2},\\
g^\prime(u)&=\frac{-(1-u^2)T_1(R_1+ug(u)T_1)+T_0(R_0+ug(u)T_0)}{(1-u^2)(T_0^2-(1-u^2)T_1^2)}.
\end{align*}
After changing the time in such a way that the denominators are removed we can write the above system  as the
differential system of the statement. From Theorem~\ref{teo:sisf'g'} we also  know that $\gamma(s)$  is a solution of system~\eqref{sis:3d}.  Finally, since the  periodic orbit is star-like it holds that $T_0+vT_1$ (that is $\dot\theta$) does not vanish on it. Thus $T_0^2\ne v^2T_1^2=(1-u^2)T_1^2.$ This fact  implies that~\eqref{sis:3d} has no critical points in the strip $\{(f,g,u)\,:\, -1<u<1\}.$ As a consequence, $\gamma(s)$ is a heteroclinic orbit of~\eqref{sis:3d} with the behavior given in the statement.
\end{proof}

 The above theorem proves that star-like periodic orbits of  \eqref{ecpol} can also be seen as heteroclinic solutions of \eqref{sis:3d} connecting a critical point of the invariant plane $\{u=-1\}$ with a critical point  of the invariant plane $\{u=1\}.$ Notice that both planes contain continua of critical points of the form  $(f,g,\pm1)$ with $f$ and $g$ satisfying $T_0(R_0\pm g T_0)=0,$ with these functions evaluated at $(f,g,\pm1).$

\subsection{A couple of  examples.}

Consider the rigid  ($\dot \theta \equiv 1$) cubic system
\begin{equation}\label{eq:ex1}
\begin{cases}
\dot{x}=-y+x(3x^2+2xy+y^2-1),\\
\dot{y}=\,x+y(3x^2+2xy+y^2-1).
\end{cases}
\end{equation}
Our associated  3$d$-differential system~\eqref{sis:3d} is
\begin{align*}
	\dfrac{{\rm d}f}{{\rm d}s}&=(1-u^2)\left(2uf^3-g-(1-u^2)(2u^2-3)g^3+6u(1-u^2)fg^2-3(2u^2-3)f^2g\right),\\
		\dfrac{{\rm d}g}{{\rm d}s}&=-f-(2\,{u}^{2}-3)f^{3}+ug+(1-{u}^{2})\big(2\,u(1-{u}^{2})g^{2}-3\,(
	2\,{u}^{2}-3)fg+6\,uf^{2}\big)g,\\
		\dfrac{{\rm d}u}{{\rm d}s}&=1-{u}^{2}.
	\end{align*}
It has the invariant  curve
\[
1+(2u^2-3)f^2=0, \quad g=0,
\]
that gives rise to the  solution $f(u)=(3-2u^2)^{-1/2}$ and $g=0$ of the corresponding system~\eqref{sis:f'g'} and to the heteroclinic orbit   of the above $3d$-differential system
\[(f(s),g(s),u(s))= \Big(\big(3-2U^2(s)\big)^{-1/2},0,U(s)\Big),\quad\mbox{with}\quad U(s)={\rm tanh}\, (s),\] joining  $(1,0,-1)$ and  $(1,0,1).$ The corresponding  limit cycle of the original planar system, expressed in polar coordinates, is 
$r=F(\theta)=(3-2\sin^2\theta)^{-1/2}.$ 
In fact, system~\eqref{eq:ex1} 
in polar coordinates correspond to  the  Bernoulli equation 
\begin{equation*}
\frac{{\rm d}r}{{\rm d}\theta}=-r+\big(2\cos^2\theta+2\cos\theta\sin\theta+1\big)\,r^3
\end{equation*}
and all its solutions are
$r=0$ and  $ r=\pm (3-2\sin^2\theta+k{\rm e}^{2\theta})^{-1/2}.$

%\frac{{\rm e}^s-{\rm e}^{-s}} {{\rm e}^s+{\rm e}^{-s}}

The next example 
\begin{equation}\label{eq:ex2}
\begin{cases}
\dot{x}=-y+x(x^2-y^2),\\
\dot{y}=x+y(x^2-y^2)\end{cases}
\end{equation}
is also a rigid cubic system, but  with a continuum of periodic orbits, because as we will see it has a center. It is one of the systems appearing in \cite{Alw2009}. 

Its  associated $3d$-differential system~\eqref{sis:3d} is
\begin{equation*}
\begin{cases}
\dfrac {{\rm d}f}{{\rm d}s}=(1-u^2)(1-2u^2)\big(3f^2+(1-u^2)g^2\big)g,\\[0.2cm]
\dfrac {{\rm d}g}{{\rm d}s}=ug+(1-2\,{u}^{2})\big(f^{2}+3(1-u^2)g^{2}\big)f,\\[0.2cm]
\dfrac {{\rm d}u}{{\rm d}s}=1-{u}^{2}.
\end{cases}
\end{equation*}

 Recall that given an $n$-dimensional polynomial vector field $Z$ it is said that a hypersurface $H=0$ is invariant if $Z(H)=KH$ for some polynomial $K,$ called the cofactor of $H,$ see for instance~\cite{LZ2012}.
  Clearly these hypersurfaces are invariant by the flow of $Z.$ It is not difficult to see that
$$H(f,g,u):=2fg-2u(f^2-(1-u^2)g^2)^2=0$$ is an invariant surface with cofactor $K(f,g,u)=u+8(1-u^2)(1-2u^2)fg$, see Figure~\ref{fig:1}. This surface is full of heteroclinic orbits connecting the planes $u=-1$ and $u=1$ that correspond to the periodic orbits of the center.

\begin{figure}[h]
	\includegraphics[scale=0.63,angle=0]{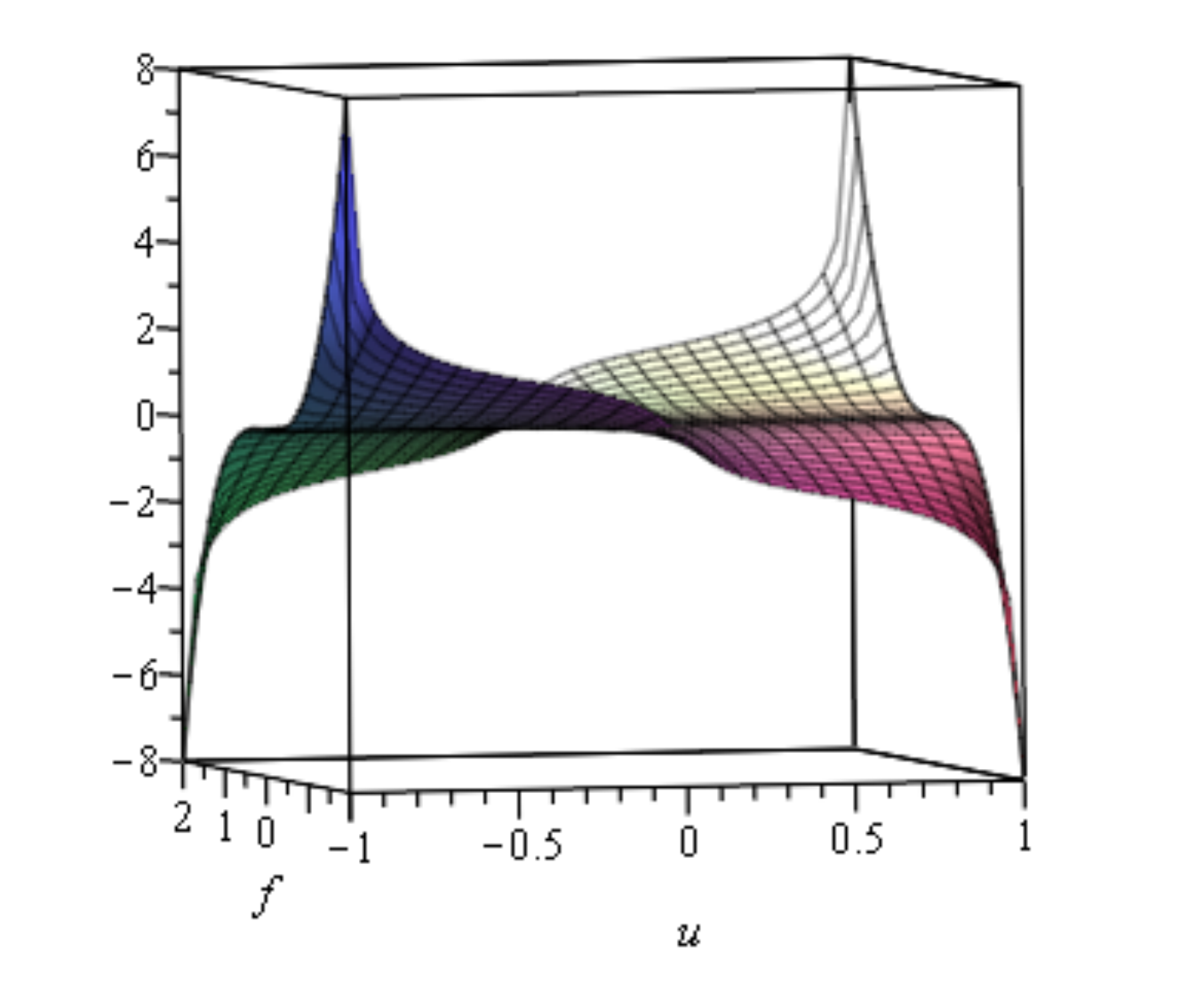}
\caption{Continua of heteroclinic solutions contained in $H(f,g,u)=0$.}\label{fig:1}
\end{figure}

Notice that system~\eqref{eq:ex2} has a center at the origin because it writes in polar coordinates as
\begin{equation*}
\frac{{\rm d}r}{{\rm d}\theta}=\big(\cos^2\theta-\sin^2\theta\big)\,r^3
\end{equation*}
and its solutions are $r=0$ and
$r(\theta)=\pm(k-\sin(2\theta))^{-1/2},$
and they are $2\pi$-periodic for $k>1.$

\section{Common solutions of two polynomial differential systems}\label{se:com}

In this section we describe a method in six steps to find conditions for two families of differential systems
\begin{align}\label{sp1}
\begin{array}{cc}
a)\, \begin{cases}
\dot{f}=P_1(f,u),\\
\dot{u}=Q_1(f,u),
\end{cases}
&\phantom{aaaaaaaaaaa}
b)\, \begin{cases}
\dot{f}=P_2(f,u),\\
\dot{u}=Q_2(f,u),
\end{cases}
\end{array}
\end{align}
where $P_i(f,u)$ and $Q_i(f,u)$ are polynomial with respect to $f$, $u$ and some parameters, for $i=1,2,$ to have a common solution.

\begin{itemize}
	\item[1.] Write systems (\ref{sp1}$a$) and (\ref{sp1}$b$) as the rational differential equations
	\begin{align}\label{sr1}
	a)\, f^{\prime}(u)=\frac{P_1(f,u)}{Q_1(f,u)},\qquad \mbox{and} \qquad b)\, f^{\prime}(u)=\frac{P_2(f,u)}{Q_2(f,u)}.
	\end{align}
	\item[2.] Consider the polynomial
	$$\mc{H}(f,u):= P_1(f,u)Q_2(f,u)-P_2(f,u)Q_1(f,u),$$
	given by 
	the
	numerator of $\frac{P_1(f,u)}{Q_1(f,u)}-\frac{P_2(f,u)}{Q_2(f,u)}$.
	\medskip
	\item[3.]  Calculate the derivative of $\mc{H}(f,u)$ with respect to $u$ where $f=f(u)$ satisfies either equation (\ref{sr1}$a$) or (\ref{sr1}$b$)
	$$\mc{J}_i(f,u):=\frac{\partial \mc{H}}{\partial{f}}(f,u)P_i(f,u)+\frac{\partial \mc{H}}{\partial{u}}(f,u)Q_i(f,u),$$
	where $i=1$ if $f(u)$ satisfies (\ref{sr1}$a$) and $i=2$ in the other case.
\item[4.] Compute the resultant between $\mc{H}(f,u)$ and $\mc{J}_i(f,u)$ with respect to $u$
$${\mc{G}}_i(f):=\Res_u\big(\mc{H}(f,u),\mc{J}_i(f,u)\big),$$
where ${\mc{G}}_i(f)$ is a polynomial that depends on $f$ and the parameters of the equations. 
We can also compute the resultant with respect to $f$ and, in this case, we  would obtain a polynomial that depends on $u$ and the parameters.
\item[5.] For some $i=1$ or $2$ solve the  algebraic system obtained by imposing that all the coefficients 
of~${\mc{G}_i}(f)$ vanish.
\medskip
\item[6.] A necessary condition for systems (\ref{sp1}$a$) and (\ref{sp1}$b$) to have a  common solution is that \[\Res_u(\mc{H}(f,u),\mc{J}_i(f,u))\equiv0\quad \mbox{and}\quad \Res_f(\mc{H}(f,u),\mc{J}_i(f,u))\equiv0.\]
\end{itemize}

 So, the solutions of the algebraic system of the  step 5 provide  the conditions on the parameters that can give rise to common solutions.

\smallskip

Let us illustrate our approach with an example. 
Consider 
\begin{align*}
\begin{array}{cc}
a)\, \begin{sistema}
\dot{f}=f^3+\lambda u-u^2+1,\\
\dot{u}=f^4-f u^2+3f^2+f,
\end{sistema}
&\phantom{aaaaaaaaaaa}
b)\, \begin{sistema}
\dot{f}=\rho f^4 u-f u^3 +f u+2u,\\
\dot{u}=f^3+3f^2-u^2+\sigma.
\end{sistema}
\end{array}
\end{align*}
We want to find conditions on the parameters $\lambda$, $\rho$ and $\sigma$ in order that the above systems can have a common solution.

Following the steps described above, we have the associated rational differential equations
\begin{align*}
f^{\prime}(u)=\frac{f^3+\lambda u-u^2+1}{f^4-f u^2+3f^2+f},\qquad f^{\prime}(u)=\frac{\rho f^4u-u^3 f+f u+2u}{f^3+3f^2-u^2+\sigma},
\end{align*}
respectively.
Then 
\begin{align*}
\mc{H}(f,u)&=-\rho f^{8}\,u+(1-3\,\rho\,u) f^{6}+(3-(\rho+1)u+
(\rho+1) u^{3}) f^{5}-2\,f^{4}u\\&
+(1+\sigma+(\lambda-3)u-2\,u^{2}+3\,u^{3}) f^{3}+(3+( 3\,\lambda-7) u-3\,u^{2}\\&
+2\,u^{3}-u^{5}) f^{2}-2u(1-u^{2})f +
 ( \sigma-{u}^{2}) ( \lambda\,u-{u}^{2}+1)
\end{align*} 
and
\begin{align*}
\mc{J}_2(f,u)&=-\rho(1+8\,\rho\, u^{2})f^{11}-3\,\rho f^{10}
-3\,\rho(6\,\rho\, u^{2}-2\,u+1)f^{9}+\big(\rho\,(5\,\rho+13)u^{4}\\&+(3-9\,\rho-5\,\rho^{2})u^{2}+\rho(15\,u-\sigma-10)-1)f^{8}-(5+3\,\rho-(9-15\,\rho)u^{2})f^{7}\\&
+(27\,\rho\,u^{4}-6(\rho+1)u^{3}+(3\,\rho\,\lambda-24\,\rho+9)u^{2}
+(3\,\rho\,\sigma+3\,\rho+2)u-3\,\rho\,\sigma\\&+\lambda-9)f^{6}+
(-(7\,\rho+5)u^{6}+(11\,\rho+2)u^{4}-(6\,\rho+15)u^{3}\\&
+(6\,\rho\,\lambda+3\,\rho\,\sigma-54\,\rho+3\,\sigma+29)u^{2}+(6\,\rho+9)u-\rho\,\sigma
+6\,\lambda-\sigma-16)f^{5}\\&+((12\,\rho+3)u^{4}-4(3\,\rho-2)u^{2}+12\,u+9\,\lambda-2\,\sigma-23)f^{4}
+(-9\,u^{6}+6\,u^{5}\\&+(9-3\,\lambda)u^{4}-(3\,\sigma+1)u^{3}+(9\,\sigma-\lambda-4)u^{2}
+(1-3\,\sigma)u+2\,\lambda\,\sigma\\&-3\,\sigma-6)f^{3}+(2\,u^{8}-u^{6}+6\,u^{5}-(6\,\lambda+5\,\sigma-30)u^{4}
-6\,u^{3}+(6\,\sigma-25)u^{2}\\&-6(\,\sigma-1)u+\sigma\,
(6\,\lambda-7))f^{2}+(-6\,u^{6}+6\,u^{4}-12\,u^{3}
+(12\,\lambda+6\,\sigma-28)u^{2}\\&+12\,u-2\,\sigma)f-4\,u^{5}+(3\,\lambda+4)u^{4}
+(6\,\sigma+2)u^{3}-(4\,\lambda\,\sigma+4)u^{2}\\&-2\,\sigma\,(\sigma+1)u+\lambda\,{\sigma}^{2}.
\end{align*}
Computing the resultant between $\mc{H}(f,u)$ and $\mc{J}_2(f,u)$ with respect to $u$ we obtain
$${\mc{G}_2}(f)=-f^2(f^3+3f^2+\sigma)(\rho f^4-f^4-3f^3-\sigma f+f+2)^2\mc{G}(f),$$
where $\mc{G}(f)$ is a polynomial of degree 64 in $f$  and whose coefficients depend on the parameters $\lambda$, $\rho$ and $\sigma$.

The algebraic system obtained  imposing  that these coefficients identically vanish is a simple over-determined system. Its unique solution is $\rho=1$, $\lambda=2$ and $\sigma=1$.  For these values of the parameters,
\[
\mc{H}(f,u)=(f^3-u^2+1)\big(-{f}^{5}u+{f}^{2}{u}^{3}-3\,{f}^{3}u+{f}^{3}-{f}^{2}u+3\,{f}^{2}-2\,uf
-{u}^{2}+2\,u+1
\big).
\]
 Each of the above two factors   gives a candidate to  be 
a common solution of both differential systems. Direct computations show that the only common solution is~$f^3-u^2+1=0$.

\section{Reversible limit cycles}\label{se:rlc}

\begin{proof}[Proof of Theorem \ref{te:stl}] By means of a rotation, if necessary, reversible star-like periodic orbits  are 
	invariant by the change  of
variables and time $(x,y,t)\rightarrow (-x,y,-t).$ This fact is equivalent to find solutions $r=F(\theta)=f(\sin \theta),$ that is with  $g=0$ in system~\eqref{sis:f'g'} in Theorem~\ref{teo:sisf'g'}. Both equations write as
\begin{equation}\label{eq:dues0}
(1-u^2)T_1\,f^\prime(u)-R_0=0,\qquad T_0\,f^\prime(u)-R_1=0.
\end{equation}
Thus we are looking for a common solution of them,
where $T_i=T_i(f,0,u)$ and $R_i=R_i(f,0,u)$, for $i=0,1$. Equivalently,
\begin{equation}\label{eq:dues}
f^\prime(u)=\frac{R_0}{(1-u^2)T_1},\qquad f^\prime(u)=\frac{R_1}{T_0}
\end{equation}
and this question can be tackled with the method that we have introduced above. In particular we have seen that these periodic orbits	are contained in 
\begin{equation}\label{eq:alc}
\mc{H}(f,u):=R_0(f,0,u)T_0(f,0,u)-(1-u^2)R_1(f,0,u)T_1(f,0,u)=0.
\end{equation}
Moreover,  when $\mc{H}(f,u)\not\equiv0,$ they are isolated periodic orbits and, so,  they are limit cycles.	

Notice that $f=r$ and $u=y/r$. 		Hence taking the numerator of $\mc{H}(f,u)=\mc{H}(r,y/r)=0$ we obtain a polynomial $F$ such that the star-like limit cycles are included in $F(r,y)=0.$ We can write $F(r,y)= r G(r^2,y)+ H(r^2,y),$
	for some new polynomials $G$ and $H$. Hence the limit cycles are contained in the algebraic curve
$	r^2G^2(r^2,y)-H^2(r^2,y)=0.$
\end{proof}

\begin{remark} Notice that Theorem~\ref{te:stl} proves that all star-like reversible limit cycles are algebraic. This is no more true for star-like reversible continua of periodic orbits.  This is so because for reversible centers it can be seen that  $R_0=T_1=0$ (see the proof of forthcoming Lemma~\ref{le_f'rev}) and the first equation in~\eqref{eq:dues0} is identically satisfied.  For instance, examples of this situation can be found inside the quadratic reversible centers, also  called Loud centers.
\end{remark}

We end this section with a couple of examples.

Consider the following system
\begin{equation*}
\begin{cases}
\dot{x}=a_1y+a_2xy^2,\\
\dot{y}=b_1 x+b_2 y+ b_3 x^3+b_4 x^2 y+ b_5y^3.
\end{cases}
\end{equation*}
By following the proof of Theorem \ref{te:stl}, its star-like reversible limit cycles are contained in \eqref{eq:alc}, that is  $\mc{H}(f,u)=R_0T_0-(1-u^2)R_1T_1,$
where
\begin{align*}
R_0&=\big((a_2+b_4)(1-u^2)+b_5u^2\big)u^2f^3+b_2u^2f,\\
R_1&= b_3u(1-u^2)f^3+u(a_1+b_1)f,\\
T_0&=b_3(1-u^2)^2f^2-a_1u^2+b_1(1-u^2),\\
T_1&=b_2u-\big(a_2u^2-b_5u^2-b_4(1-u^2)\big)uf^2.
\end{align*}
Thus 
 $$\mc{H}(f,u)=u^2f\Big(-a_1b_2+\big((1-u^2)(a_2b_1-a_1b_4) -a_1b_5u^2\big)f^2+(1-u^2)^2a_2b_3f^4\Big).$$
  A similar expression to $\mc{J}_2(f,u),$  but calculating   the derivative of function $\widetilde{ \mc{H}}(f,u)=\mc{H}(f,u)/(u^2f)$ with respect to $u$ where $f=f(u)$ satisfies the second differential equation in~(\ref{eq:dues}), instead of the derivative of $\mc{H}(f,u)$, is
 $$\widetilde{\mc{J}}_2(f,u)=2a_1uf^2\big(b_3(1-u^2)(2a_2-b_5)f^2- a_1b_4+a_2b_1-b_1b_5\big).$$
Following the steps of the algorithm described in Section \ref{se:com} we perform the resultant between $\widetilde{\mc{H}}$ and ${{\widetilde{\mc{J}}_2}}/{(uf^2)}$ with respect to $f.$
We  obtain the polynomial
$$G(u)=16\,a_1^4\,b_3^2\,(1-u^2)^2 H^2(u),$$%16\,a_1^8\,b_2^4\,b_3^2
where
\[
\begin{split}
H(u)=&\big(a_1^2a_2b_4^2-2a_1^2a_2b_4b_5-a_1^2b_4^2b_5 +a_1^2b_4b_5^2 -2a_1a_2^2b_1b_4-a_2^2b_1^2b_5+2a_1a_2^2b_1b_5\\&+4a_1a_2^2b_2b_3
+2a_1a_2b_1b_4b_5 -3a_1a_2b_1b_5^2+a_1b_1b_5^3+a_2^3b_1^2-4a_1a_2b_2b_3b_5\\&-a_1b_1b_4b_5^2+a_1b_2b_3b_5^2
\big)u^2-a_1^2a_2b_4^2+a_1^2b_4^2b_5+2a_1a_2^2b_1b_4 -4a_1a_2^2b_2b_3\\&-2a_1a_2b_1b_4b_5+4a_1a_2b_2b_3b_5
+a_1b_1b_4b_5^2-a_1b_2b_3b_5^2-a_2^3b_1^2+a_2^2b_1^2b_5.
\end{split}
\]
Imposing that the coefficients with respect to $u$ of $H(u)$ vanish we obtain  several sets of solutions, that might indicate the existence of periodic orbits in the differential system. One of these sets is $a_2 = b_5/2$, $b_1= -2a_1b_4/b_5$, and $a_1, b_2, b_4, b_3,$  $b_5\ne0$ free parameters.

For instance  taking $a_1=-1$, $b_2=1$, $b_4=0$, $b_3=1$, $b_5=-2$ we obtain  the system 
\begin{align}\label{sys:cla}
\begin{cases}
\dot{x}=-y(1+xy),\\
\dot{y}=y+x^3-2y^3.
\end{cases}
\end{align}
 The corresponding differential equations \eqref{sp1} obtained from \eqref{sis:f'g'} are
\begin{align*}
\begin{array}{cc}
a) \begin{cases}
\dot{f}=-uf(f^2u^2+f^2-1),\\
\dot{u}=(1-u^2)(1-f^2u^2),
\end{cases}
&\quad
b) \begin{cases}
\dot{f}=uf\big((1-u^2)f^2-1\big), \\
\dot{u}=(1-u^2)^2f^2+u^2.
\end{cases}
\end{array}\phantom{aaaaaa}
\end{align*}
They share the algebraic curve $(1-u^2)^2f^4+2f^2u^2-1=0$ that gives rise to  
the  algebraic star-like reversible limit cycle 
  $x^4+2y^2-1=0,$ with cofactor $-4y^2.$ Moreover, since this cofactor does not change sign we can prove that system~\eqref{sys:cla} has no more limit cycles, algebraic or not. This fact follows by using the ideas introduced in~\cite[p. 92]{CGL}.

\medskip
Another  set of suitable values of the parameters  is
$a_1=20$, $a_2=2,$ $b_1=20$,  $b_3=-20$, $b_4=-2$, $b_5=4$  and $b_2$ free. This gives 
\begin{align}\label{sys:cla2}
\begin{cases}
\dot{x}=2y(10+xy),\\
\dot{y}=20x+b_{2}y-20x^3-2x^2y+4y^3,
\end{cases}
\end{align}
and the corresponding differential equations \eqref{sp1} obtained from \eqref{sis:f'g'}  are
\begin{align*}
\begin{array}{cc}
a) \begin{cases}
\dot{f}=uf\big(4u^2f^2+b_{2}\big),\\
\dot{u}=(1-u^2)(4u^2f^2-2f^2+b_{2}),
\end{cases}
&\!\!\!
b) \begin{cases}
\dot{f}=uf\big((1-u^2)f^2-2\big), \\
\dot{u}=(1-u^2)^2f^2+2u^2-1.
\end{cases}
\end{array}
\end{align*}
By applying our approach we get the common solution to both systems  which is given by
the curve
$$2(1-u^2)^2f^4+4(2u^2-1)f^2+b_{2}=0.$$ 
 In cartesian coordinates it corresponds to the invariant algebraic curve $2x^4+4y^2-4x^2+b_2=0,$ with cofactor $8y^2.$   It can be seen that for $b_{2}<0$ it gives rise to an algebraic star-like  reversible  limit cycle, that surrounds three critical points. When $b_{2}=0$ it gives a double heteroclinic loop, while for  $0<b_{2}<2$ provides two algebraic limit cycles, one being the mirror image of the other with respect to the line $x=0,$ and each one of them surrounding a unique critical point. Again, as in system~\eqref{sys:cla}, the fact that the cofactor does not change sign implies that the only limit cycles of  system~\eqref{sys:cla2}, are the ones contained in this algebraic curve. This example recovers previous results in~\cite{LliZha2007}.

\section{Bifurcation of limit cycles from reversible centers}\label{se:abe}

The goal of this section is to prove Theorem~\ref{te:N(a)} and give some examples of application.

\begin{proof}[Proof of Lemma~\ref{le:des}] For a general $H$ the proof simply follows  using that
$v^{2k}=(1-u^2)^k$ and $v^{2k+1}=(1-u^2)^k v$  in each of its
monomials. For $H(x,y)=K(x^2,y)$ simply write
\begin{align*}
H(x^2,y)&=H\big((f+vg)^2v^2, (f+vg) u\big)\\&=
 H\big((1-u^2)f^2+(1-u^2)^2 g^2+2(1-u^2)vfg, u f + u v g\big)
\end{align*}
and perform the (finite) Taylor expansion at the point
 $\big((1-u^2)f^2,u f\big).$
\end{proof}

For short, we will write the decomposition introduced in Lemma~\ref{le:des}
as
\[
H= H_0+v H_1\quad\mbox{and}\quad K=K_{0,0}+
(1-u^2)g^2 K_{0,1}+ v g K_{1,0}.
\]

\begin{lemma}\label{le_f'rev}
The star-like periodic orbits of
\begin{align*}
\dot{x}= A(x^2,y),\qquad\qquad \dot{y}=x B(x^2,y),
\end{align*}
passing through the point $(0,\rho),$ $\rho>0,$ where $A$ and $B$ are  polynomials,  are given by $r=f_\rho(\sin
\theta)$ where $f=f_\rho(u)$ is the solution of the first order
Cauchy initial value problem~\eqref{f'_rev}:
\begin{equation*}
f^{\prime}(u)=\frac{\left(A_{0,0}+u  B_{0,0} f(u) \right)
f(u)}{(1-u^2)B_{0,0} f(u) -uA_{0,0}}, \quad f(0)=\rho,
\end{equation*}
where $A_{0,0}=A\big((1-u^2)f^2,uf\big)$ and
$B_{0,0} =B\big((1-u^2)f^2,u f\big)$.
\end{lemma}

\begin{proof}  By using Theorem~\ref{teo:sisf'g'}  with $g(u)=0$, the two differential equations~\eqref{sis:f'g'} write as  
\begin{equation*}
\begin{cases}
(1-u^2)T_1f^\prime(u)-R_0=0,\\
T_0f^\prime(u)-R_1=0.
\end{cases}
\end{equation*}	
Some computations give that 
\begin{align*}
R_0=0,\, R_1=A_{0,0}+uf(u)B_{0,0},\,
T_0=(1-u^2)B_{0,0}-uA_{0,0}/f(u)\,\,\,\mbox{and}\,\,\, T_1=0.
\end{align*}
Therefore, the first equation of the above system is identically satisfied and the second one 
gives the desired differential equation.
\end{proof}

\begin{proof}[Proof of Theorem \ref{te:N(a)}]
	
Consider system
\begin{align*}
	\begin{cases}
		\dot{x}= X(x,y;\varepsilon)= c(x,y)+\ep C(x,y)= A(x^2,y)+\ep C(x,y),\\
		\dot{y}=Y(x,y;\varepsilon)=d(x,y)+\ep D(x,y)=x B(x^2,y)+\ep
		D(x,y).
	\end{cases}
\end{align*}
By Theorem~\ref{teo:sisf'g'} we have to find solutions of system~\eqref{sis:f'g'} associated to the above planar system that are well defined in $[-1,1].$ After some computations we get that the functions appearing in~\eqref{sis:f'g'} are:
\begin{align*}
R_0=&(1-u^2)X_1+u Y_0,\\
	R_1=& X_0+uY_1,\\
	T_0=&\frac{(u^2Y_1+uX_0-Y_1)f-(1-u^2)(uX_1-Y_0)g  }{(1-u^2)g^2-f^2},\\
	T_1=&\frac{(uX_1-Y_0)f-(u^2Y_1+uX_0-Y_1)g}{(1-u^2)g^2-f^2}  ,
\end{align*}
where for each $\varepsilon$ we have decomposed $X$ and $Y$ as in Lemma~\ref{le:des}. Equivalently,
\begin{align*}
	R_0=&(1-u^2)c_1+u d_0+\varepsilon \left((1-u^2)C_1+u D_0\right) ,\\
	R_1=& c_0+ud_1 +\varepsilon \left( C_0+uD_1\right)
\end{align*}	
and
\begin{align*}	
	T_0=&\frac{(u^2d_1+uc_0-d_1)f-(1-u^2)(uc_1-d_0)g  }{(1-u^2)g^2-f^2}\\&+\varepsilon\frac{(u^2D_1+uC_0-D_1)f-(1-u^2)(uC_1-D_0)g  }{(1-u^2)g^2-f^2},\\
	T_1=&\frac{(uc_1-d_0)f-(u^2d_1+uc_0-d_1)g}{(1-u^2)g^2-f^2}\\& +\varepsilon \frac{(uC_1-D_0)f-(u^2D_1+uC_0-D_1)g}{(1-u^2)g^2-f^2}.
\end{align*}
Using that $c(x,y)=A(x^2,y)$ and $d(x,y)=xB(x^2,y)$ and again the notation of~Lemma~\ref{le:des} we have that
\begin{align*}
c=&c_0+vc_1= A_{0,0}+ (1-u^2) g^2 
A_{0,1}+vg A_{1,0},\\
d=& d_0+vd_1=(f+vg)v\big( B_{0,0}+ (1-u^2) g^2 
B_{0,1}+vg B_{1,0}\big)=\\=&B_{0,1}(1-u^2)^2g^3+(1-u^2)(B_{1,0}f+B_{0,0})g+v\big((1-u^2)(B_{0,1}f+B_{1,0})g^2+B_{0,0}f \big).
\end{align*}
Replacing these equalities  in the above ones   we get that
\begin{align*}
	R_0=& B_{0,1} (1-u^2)^2u g^3+(1-u^2)(B_{1,0}fu+B_{0,0}u+A_{1,0})g	
	\\&+\ep\big( (1-u^2)C_1+uD_0\big),\\
	R_1=& (1-u^2)(B_{0,1}fu+B_{1,0}u+A_{0,1})g^2+B_{0,0}uf+A_{0,0}
			\\&+\ep\big(C_0+uD_1 \big)
\end{align*}	
and
\begin{align*}
	 	T_0=& \frac{B_{0,1}(1-u^2)^3g^4+ (B_{0,0}- B_{0,1}f^2)(1-u^2)^2g^2 }{(1-u^2)g^2-f^2}\\ 
	 	&+\frac{(1-u^2)(A_{0,1}f-A_{1,0})ug^2-B_{0,0}(1-u^2)f^2+A_{0,0}uf}{(1-u^2)g^2-f^2}
\\&+\varepsilon\frac{(u^2D_1+uC_0-D_1)f-(1-u^2)(uC_1-D_0)g  }{(1-u^2)g^2-f^2},\\
	T_1=&\frac{ 	
B_{1,0}(1-u^2)^2g^3-(1-u^2)(A_{0,1}ug^3+B_{1,0}f^2g)+ug(A_{1,0}f-A_{0,0})
}{(1-u^2)g^2-f^2} \\&	
 +\varepsilon \frac{(uC_1-D_0)f-(u^2D_1+uC_0-D_1)g}{(1-u^2)g^2-f^2}.
\end{align*}

We look for solutions of the form
$r=F(\theta,\ep)=f(u,\ep)+v g(u,\ep),$ where
\begin{align*}
f(u,\ep)&=f_{\rho}(u)+\ep P(u)+O(\ep^2),\\
g(u,\ep)&=\ep Q(u)+O(\ep^2)
\end{align*}
and recall that  $f_\rho(u)$ is given in Lemma~\ref{le_f'rev}.

Next we plug the above expressions of $T_0,T_1,R_0,R_1,$ $f(u,\ep)$ and $g(u,\ep)$ in system~\eqref{sis:f'g'}.	After some tedious but straightforward computations we get that the coefficient with respect to $\ep$  of order one of the  equation corresponding to $g$ gives
\begin{equation*}
Q^\prime(u)=L_{\rho}(u)Q(u)+M_{\rho}(u),
\end{equation*}
where $L_{\rho}$ and $M_\rho$ are as in the statement of the theorem.
The general solution of this first order linear differential equation is
$$
Q(u)=
{\rm e}^{\int_0^uL_{\rho}(s){\rm d}s}\left(m+\int_0^uM_{\rho}(w){\rm e}^{-\int_0^w
	L_{\rho}(s){\rm d}s}{\rm d}w\right),
$$
 where $m$ is a constant of integration. Since ${\rm e}^{\int_0^uL_{\rho}(s){\rm d}s}$ diverges for $u=\pm 1$,  due to the singularities at the denominator of $L_\rho$ at $u=\pm1,$  we need two conditions for  $Q(u)$ to be  well defined at these two points:	
$$
m+\int_0^1M_{\rho}(u){\rm e}^{-\int_0^u L_{\rho}(s){\rm d}s}{\rm d}u=0\quad\mbox{and}\quad m+\int_0^{-1}M_{\rho}(u){\rm e}^{-\int_0^u
	L_{\rho}(s){\rm d}s}{\rm d}u=0.
$$
Joining these equations we obtain that a necessary condition for the persistence of a perturbed periodic orbit is that
$$N(\rho):=\int_{-1}^1M_{\rho}(u){\rm e}^{-\int_0^u L_{\rho}(s){\rm d}s}{\rm d}u=0,$$
at this value of $\rho,$ as we wanted to prove.
\end{proof}

\noindent We end the paper with a couple of   applications of Theorem~\ref{te:N(a)}.

 Consider  the system
\[\begin{cases}
\dot{x}=y-x^2y+\ep(a_{1}x+a_{2}xy^2+a_{3}x^3),\\
\dot{y}=-x-xy^2+\ep(b_{1}y+ b_{2}x^2y+b_{3}y^3).\end{cases}
\]
 When  $\epsilon=0,$ it has the first integral $H(x,y)=(y^2+1)/(x^2-1)$ and a center at the origin.

We apply Theorem~\ref{te:N(a)} 
with
\begin{align*}
L_{\rho}(u)=\frac{u(1+3(1-u^2)f_{\rho}^2(u))}{1-u^2},\quad{\rm and}\quad
M_{\rho}(u)=\frac{P_\rho(u)}{1-u^2},
\end{align*}
where
\begin{align*}
P_\rho(u)&=\,u^2(1-u^2)\big(b_2-a_{3}+(a_{3}-a_{2}+b_{3}-b_{2})u^2\big)f_{\rho}^5(u)
\\&+\big(-b_{3}u^4+(1-u^2)(-a_{3}+(a_{3}-a_{1}-a_{2}+b_{1}-b_{2})u^2 \big)f_{\rho}^3(u)\\&+\big(-a_{1}+(a_{1}
-b_{1})u^2\big)f_{\rho}(u).
\end{align*}

Moreover, in this case $f_{\rho}(u)$ is the solution of 
$
f^{\prime}=uf^3,$ with $f(0)=\rho.
$
So, $f_{\rho}(u)=\rho/\sqrt{1-\rho^2u^2}.$ As a consequence, 
we have that
\begin{align*}
{\rm e}^{-\int_0^u 	L_{\rho}(s){\rm d}s}=&{\rm e}^{\int_0^u \frac{s(4\rho^2 s^2-3\rho^2-1)}{(1-\rho^2s^2)(1-s^2)}{\rm d}s}
=(1-\rho^2 u^2)^{3/2}(1-u^2)^{1/2},\\
M_{\rho}(u)&=\frac{Q_\rho(u)}{(1-\rho^2u^2)^{5/2}(1-u^2)},
\end{align*}
where
\begin{align*}
Q_\rho(u)=&(b_{3}-b_{1}-b_{2})u^4\rho^5+b_{2}u^2\rho^5+(a_{2}-a_{1}-a_{3}+b_{1}+b_{2}-b_{3})u^4\rho^3\\&+(a_{1}-a_{2}+2a_{3}+b_{1}-b_{2})u^2\rho^3-a_{3}\rho^3+(a_{1}-b_{1})u^2\rho-a_{1}\rho.
	\end{align*}
	
After some more computations  we obtain
\[
N(\rho)=\int_{-1}^1\frac{Q_\rho(u)}{(1-\rho^2u^2)(1-u^2)^{1/2}}{\rm d}u=\frac{ \pi q(\rho)}{2\rho},
\]
where
\begin{align*}
	q(\rho)=&(b_{2}-b_{1}+b_{3})(1-\rho^2)^2+2(a_{3}-b_{2})(1-\rho^2)^{3/2}-(a_{1}-a_{2}+3a_{3}
	\\&-b_{1}+3b_{3}
	-b_{2})(1-\rho^2)-2(a_{2}-b_{3})(1-\rho^2)^{1/2}+a_{1}+a_{2}+a_{3}.
	\end{align*} 
	Doing the change of variable $\lambda^2=1-\rho^2$ and by using that $\lambda-1$ is a factor of $q(\rho)$ the positive zeroes of $N$ can be found from the solutions of
\begin{equation*}
	(b_{2}-b_{1}+b_{3})\lambda^3\!+\!(2a_{3}-b_{2}-b_{1}+b_{3})\lambda^2-(a_{1}-a_{2}+a_{3}
	+2b_{3})
	\lambda-a_{1}-a_{2}-a_{3}=0.
	\end{equation*} So, by Theorem~\ref{te:N(a)},  at most 3 periodic orbits can persist.

A similar result could be obtained using Abelian integrals. The above system  is already studied in~\cite{CheLiLliZh2002}.

As a second example, consider  the system
\[\begin{cases}
\dot{x}=y-y^2+\ep(ax+bxy),\\
\dot{y}=-2x.\end{cases}
\]
The unperturbed system is hamiltonian, with hamiltonian function
\[
H(x,y)=x^2+\frac{y^2}{2}-\frac{y^3}{3}
\] 
and  has a center at the origin. 
 In the expression of  $N(\rho)$ given in  Theorem~\ref{te:N(a)} the functions $L_\rho$ and $M_\rho$ are
\begin{align*}
L_{\rho}(u)&=\frac{u\left(6+2u^4-7u^2+(4u-10u^3+4u^5)f_{\rho}(u)+(2u^2-1)u^4f_{\rho}^2(u)\right)}{(1-u^2)(u^3f_{\rho}(u)+u^2-2)^2},\\
M_{\rho}(u)&=\frac{-2f_{\rho}(u)(a+buf_{\rho}(u))}{(u^3f_{\rho}(u)+u^2-2)^2},
\end{align*}
where $f_{\rho}$ satisfies
\begin{equation*}
f^\prime=\frac{-uf(uf+1)}{u^3f+u^2-2},\quad f(0)=\rho.
\end{equation*}
Then,
$$N(\rho)=0  \Longleftrightarrow \int_{-1}^1\frac{-2f_{\rho}(u)(a+buf_{\rho}(u))}{(u^3f_{\rho}(u)+u^2-2)^2}\,{\rm e}^{-\int_0^u L_{\rho}(s){\rm d}s}{\rm d}u=0,$$
or equivalently,
\begin{equation*}
-\frac{a}{b}=\frac{\int_{-1}^1\frac{uf_{\rho}^2(u)}{(u^3f_{\rho}(u)+u^2-2)^2}{\rm e}^{-\int_0^u L_{\rho}(s){\rm d}s}{\rm d}u}
{\int_{-1}^1\frac{f_{\rho}(u)}{(u^3f_{\rho}(u)+u^2-2)^2}{\rm e}^{-\int_0^u
		L_{\rho}(s){\rm d}s}{\rm d}u},
\end{equation*}
is a condition for a periodic orbit to persist.	

 After some tricky computations we get that
\[
-\frac{a}{b}=\frac{\int_{-1}^1\frac{uf_{\rho}^2(u)}{(u^3f_{\rho}(u)+u^2-2)^2}{\rm e}^{-\int_0^u L_{\rho}(s){\rm d}s}{\rm d}u}
{\int_{-1}^1\frac{f_{\rho}(u)}{(u^3f_{\rho}(u)+u^2-2)^2}{\rm e}^{-\int_0^u
		L_{\rho}(s){\rm d}s}{\rm d}u}=\frac{\int_{H=h}xy\,{\rm d}y}{\int_{H=h}x\,{\rm d}y},
\]
where $h=\rho^2.$ Notice that the last fraction  corresponds to the one obtained by using the classical Abelian integrals approach to know which periodic orbits of the hamiltonian system persist. Following the results of Drachman, van Gils and Zhang (\cite{DVZ}) the above equation has at most one solution, so, at most  one periodic orbit persists.

%%%%%%%%%

%%-----------------------------------------------

\subsection*{Acknowledgements}
The first author is supported by FONDECyT grant number 11171115.
The second author is supported by 
Ministerio de Ciencia, Innovación y Universidades of the Spanish Government through grants MTM2016-77278-P
(MINECO/AEI/FEDER, UE) and  
by  grant 2017-SGR-1617  from
AGAUR,  Generalitat de Catalunya.

\end{document}